\documentclass[%
10pt,a4paper,leqno]{article}

\pdfoutput=1
\newcommand{\cri}{\\ \indent}

\usepackage[english]{babel}
\usepackage{verbatim}
\usepackage{graphicx,color}
\usepackage{enumerate}

\usepackage[intlimits]{amsmath}
\usepackage{amsthm}
\usepackage{amssymb}
\usepackage{hyperref}

\newtheoremstyle{introThm}%
     {\topsep}
     {\topsep}
     {\itshape}
     {}
     {\bfseries}
     {}
     {\newline}
     {\thmname{#1} \thmnumber{#2} \thmnote{ #3}}

\newtheoremstyle{externalThm}%
     {\topsep}
     {\topsep}
     {\itshape}
     {}
     {\bfseries}
     {}
     {0ex}
     {\thmname{#1} \thmnumber{#2}. \thmnote{ #3}}

\newtheoremstyle{hypothesis}%
     {\topsep}
     {\topsep}
     {}
     {}
     {\bfseries}
     {}
     { }
     {\thmname{#1}\thmnumber{#2} \thmnote{ #3}}

\newtheorem{new}{Propositions}[section]
\newtheorem{thm}[new]{Theorem}

\newtheorem{prop}[new]{Proposition}
\newtheorem{lemma}[new]{Lemma}
\newtheorem{cor}[new]{Corollary}

\newtheorem*{unnumberedThm}{Theorem}
\newtheorem*{unnumberedCor}{Corollary}

\theoremstyle{remark}
\newtheorem{rmk}[new]{Remark}

\theoremstyle{externalThm}
\newtheorem{externalTheorem}{Theorem}

\theoremstyle{hypothesis}

\newtheorem*{unnumberedHyp}{H}

\newcommand{\ie}{i.e.\ }
\newcommand{\eg}{e.g.\ }
\newcommand{\af}{\emph{a fortiori}}

\newcommand{\caveat}{\emph{Caveat}}

\newcommand{\bsc}{BSC}

\newcommand{\Hess}{\nabla\sp{2}}
\newcommand{\Hu}{\Hess u}
\newcommand{\diff}{\nabla}
\newcommand{\J}{\diff}

\newcommand{\ball}[2]{B(#1, #2)}
\newcommand{\clball}[2]{\overline{B}(#1, #2)}
\newcommand{\z}{z\sb0}

\newcommand{\N}{\mathbb{N}}
\newcommand{\R}{\mathbb{R}}
\newcommand{\Rr}{\R\sp{2}}
\newcommand{\Rn}{\R\sp{N}}
\newcommand{\Rhc}{[0,+\infty)}
\newcommand{\Rho}{(0,+\infty)}

\newcommand{\wq}{W\sp{1,q}(\Omega)}

\newcommand{\wqcs}{W\sp{1,q}\sb0(\Omega)}
\newcommand{\wi}{W\sp{1,\infty}(\Omega)}
\newcommand{\wss}{W\sp{2,2}}

\renewcommand{\phi}{\varphi}
\newcommand{\fun}[3]{#1\!:\,#2 \rightarrow {#3}}

\newcommand{\jj}{\mathcal{J}}

\newcommand{\ja}{\mathcal{J}\sb{\a}}

\newcommand{\f}{f}
\newcommand{\g}{g}

\newcommand{\fb}{\f'}

\newcommand{\vv}{p}
\renewcommand{\a}{n}
\newcommand{\ainv}{\tfrac{1}{\a}}
\newcommand{\fa}{\f\sb{\a}}
\newcommand{\ga}{\g\sb{\a}}
\newcommand{\adom}{\a\in\N}

\newcommand{\ua}{u\sb{\a}}
\newcommand{\va}{v\sb{\a}}

\newcommand{\Aa}{A\sb{\a}}
\newcommand{\Ba}{B\sb{\a}}

\newcommand{\aaa}{\alpha\sb{\a}}
\newcommand{\aaainfty}{\alpha\sb{\infty}}

\newcommand{\aaalim}{\alpha}
\newcommand{\bbalim}{\beta}
\newcommand{\gm}[1]{|\diff #1|}

\DeclareMathOperator{\dv}{div}
\DeclareMathOperator{\lp}{\Delta}
\DeclareMathOperator{\ii}{\mathrm{I}}
\DeclareMathOperator{\tr}{\mathrm{tr}}

\usepackage{amssymb}
\usepackage{amsfonts}
\usepackage{amsmath}
\usepackage{epsf}

\def\NN{\mathbb N}

\def\RR{\mathbb R}
\def\RN{\RR^N}
\def\Om{\Omega}
\def\pa{\partial}
\def\cA{\mathcal A}
\def\cB{\mathcal B}
\def\cK{\mathcal K}
\def\cJ{\mathcal J}
\def\cU{\mathcal U}

\def\de{\delta}

\def\pa{\partial}

\def\fhi{\varphi}
\def\tr{\mathrm{tr }}
\def\dv{\mathrm{div\,}}

\def\na{\nabla}
\def\ints{\int\limits}
\def\ds{\displaystyle}

\author{Simone Cecchini \footnotemark[1]{\ \,}\footnotemark[2] \and
  Rolando Magnanini \footnotemark[3]}
\newenvironment{acknowledgements}%
{\textbf{\large Acknoledgements. }}%
{}

\title{Critical Points of Solutions of Degenerate Elliptic Equations
  in the Plane}

\begin{document}
\renewcommand{\thefootnote}{\fnsymbol{footnote}}
\footnotetext[1]{
          Dipartimento di Matematica U.~Dini, %
          Universit\` a di Firenze, viale Morgagni 67/A, 50134
          Firenze, Italy.\ %
}
\footnotetext[2]{
          Current Address: Dipartimento di Scienze e Metodi dell'Ingegneria, 
          Universit\`a degli Studi di Modena e Reggio Emilia,
          via G. Amendola, 2 - Pad. Morselli - 42100 Reggio Emilia, Italy.\ %
          Email: simone.cecchini@unimore.it}
\footnotetext[3]{
          Dipartimento di Matematica U.~Dini, %
          Universit\` a di Firenze, viale Morgagni 67/A, 50134
          Firenze, Italy.\ %
          Email: rolando.magnanini@math.unifi.it}
\maketitle

\renewcommand{\b}{\ball{\z}{R}}
\newcommand{\bb}{\ball{\z}{r}}
\newcommand{\cbb}{\clball{\z}{r}}
\newcommand{\bh}{\ball{\z}{r}\setminus\{\z\}}
\newcommand{\B}{{B\sb{\varepsilon}}}
\newcommand{\vuz}{U\sb0}
\newcommand{\dx}{dx}
\newcommand{\dy}{dy}
\newcommand{\dw}{d\vartheta}
\renewcommand{\ds}{ds}
\newcommand{\dt}{dt}
\newcommand{\s}{\mathbb{S}\sp1}
\newcommand{\A}{\mathrm{A}}
\newcommand{\Rs}{\R\sp2}
\newcommand{\W}{\mathrm{W}\sp{1,2}}
\newcommand{\WW}{\mathrm{W}\sp{2,2}}
\newcommand{\ci}[1]{C\sb0\sp{\infty}(#1)}
\newcommand{\ls}{\mathrm{L}\sp2}
\newcommand{\T}[1]{\tau\sb{#1}}
\renewcommand{\tt}{\T{h}}
\newcommand{\e}{\mathrm{e}\sb s}
\newcommand{\id}{\mathbb{I}d}
\newcommand{\wc}{\subset\subset}
\newcommand{\annular}{A}
\newcommand{\brr}{\annular\sp{r\sb2}\sb{r\sb1}}
\newcommand{\vu}{\dfrac{\diff u}{|\diff u|}}
\newcommand{\D}[1]{\mathcal{D}(#1)}
\newcommand{\cF}{\mathcal{F}}
\newcommand{\bp}{B\sp\prime}
\renewcommand{\b}{B}
\renewcommand{\phi}{\varphi}
\newcommand{\spt}{\mathrm{supp}}
\newcommand{\ele}{Euler-Lagrange equation}
\newcommand{\domj}{ \textrm{domain of }\jj}
\newcommand{\multa}{\sigma\sb{\a}}
\newcommand{\sym}{\mathcal{S}\sp n}
\newcommand{\Hphi}{\Hess\phi}
\newcommand{\x}{\hat{x}}
\renewcommand{\B}{B}
\newcommand{\br}{\B\sb{r}}
\newcommand{\bt}{\B\sb{t}}





\thispagestyle{plain}


\begin{abstract}
  We study the minimizer $u$ of a convex functional in the plane which
  is not G\^ateaux-differentiable. Namely, we show that the set of
  critical points of any $C\sp1$-smooth minimizer can not have
  isolated points. Also, by means of some appropriate approximating
  scheme and viscosity solutions, we determine an Euler-Lagrange
  equation that $u$ must satisfy.
  By applying the same approximating scheme, we can pair $u$ with a
  function $v$ which may be regarded as the stream function of $u$ in
  a suitable generalized sense.
\end{abstract}

\newcommand{\keywords}[1]{{\noindent\scriptsize{\textbf Keywords:}\ #1}}
\newcommand{\subclass}[1]{{\noindent\scriptsize{\textbf MSC:}\ #1}}

\begin{center}\begin{minipage}[t]{0.85\linewidth}
\\
\subclass{35B05,35B38,35J20,35J60.}
\end{minipage}\end{center}

\section{Introduction}
\label{sec 1}

\subsection{Motivations: a case study.} 
This paper will mainly focus on the properties of certain convex coercive
{\it non-differentiable} functionals and their extremals. We are partly 
motivated by the investigations that the second author and G. Talenti 
pursued in a series of papers
\cite{magnaniniTalenti1}-\cite{magnaniniTalenti3} about {\it
  complex-valued} 
solutions of the classical {\it eikonal equation} in the plane.
\par
One of the main charachters acting in those papers is the functional
\begin{equation}
\label{defJ}
\jj(u)=\ints_\Om f(|\na u|)\, dx,
\end{equation}
where $\Om$ is a bounded domain in $\RN$ (therein, with $N=2$),
$|\na u|$ denotes the modulus of the gradient of a scalar function $u$
defined in $\Om$ and
\begin{equation}
\label{deff}
f(\rho)=\frac12 \left[\rho\sqrt{1+\rho^2}+\log\left(\rho+\sqrt{1+\rho^2}
\right)\right]
\ \mbox{ for } \ \rho\ge 0.
\end{equation}
\par
Notice that $f$ is strictly convex and grows quadratically at infinity,
thus the existence and uniqueness of a function minimizing $\cJ$ 
subject to a Dirichlet boundary condition is not hard to prove.
However, since $f'(0)>0,$ $\cJ$ is not always differentiable 
--- non-differentiability occurring when the Lebesgue measure of the 
set $\{x\in\Om: \na u(x)=0\}$ is positive  --- and
a standard Euler-Lagrange equation may not be available for $\cJ,$
though a differential inclusion
$$
0\in\pa\cJ(u),
$$
by means of the subdifferential $\pa\cJ,$ still characterizes a minimizing $u$
(see \cite{magnaniniTalenti2} for details). A formal Euler-Lagrange
equation would read  
\begin{equation}
\label{divFormForU}
\dv\left\{ \fb(|\na u|)\,\frac{\na u}{|\na u|}\right\}=0
\end{equation}
with $f'(\rho)=\sqrt{1+\rho^2}$ --- a clearly singular equation exactly at the {\it critical points} of $u.$ 
Nevertheless, away from its critical points, a smooth minimizer $u$ certainly
satisfies the quasilinear {\it elliptic degenerate} differential equation
\begin{equation}
\label{nondivu}
\tr [\cA(\na u)\,\na^2 u]=0
\end{equation}
(degeneration occurring, of course, at critical points), where the
$N\times N$ matrix $A(p)$ has coefficients 
$$
\cA_{ij}(p)=\big[\alpha(|p|)-1\big]p_ip_j + \,|p|^2\de_{ij},
\quad i, j=1,\dots, N
$$
Here, $\de_{ij}$ is the usual Kronecker's delta, while
\begin{equation}\label{alphaIntro}
  \alpha(\rho) = %
  \begin{cases}
    \dfrac{\rho f''(\rho)}{f'(\rho)}, & \rho>0, \\
    0, & \rho=0.
  \end{cases}
\end{equation}
For $N=2$ and $f$ given by \eqref{deff}, Equation \eqref{nondivu} may
be formally re-written as
\begin{equation}
\label{nondivu2}
-\left(|\na u|^4+u_y^2\right)\, u_{xx}+2\, u_x\, u_y\, u_{xy}
-\left(|\na u|^4+u_x^2\right)\, u_{yy}=0.
\end{equation}
\par
Functional \eqref{defJ} and equations \eqref{divFormForU}, \eqref{nondivu} show
some interesting features.
\par
Even if, as already mentioned, minimizers of $\cJ$ do not satisfy \eqref{divFormForU} in general,
it is proved in \cite{magnaniniTalenti2} that they are {\it viscosity
  solutions} of \eqref{nondivu2}.  
In \cite{magnaniniTalenti2}, it is also shown that {\it classical}
solutions of \eqref{nondivu2} 
exist which {\it can not} be minimizers of $\cJ$ --- thus proving that the Dirichlet problem
for viscosity solutions of \eqref{nondivu2} is not uniquely solvable.
\par
Another interesting feature concerns the set of critical points of solutions of \eqref{nondivu2} ---
a decisive information for a good understanding of the properties of \eqref{defJ}, \eqref{divFormForU} and \eqref{nondivu}.
For $N=2,$ sample solutions of \eqref{nondivu2} have their gradients which vanish on a
set of {\it positive} Lebesgue measure (\cite{magnaniniTalenti1});
also, it has been shown (\cite{magnaniniTalenti1}) 
that classical solutions of \eqref{nondivu2} cannot have isolated (non-degenerate) critical points,
that is their gradients either never vanish or annihilate on a continuum ---
a property not occurring for other well-known degenerate elliptic equations (e.g. the $p$-Laplace
equation).
\par
This phenomenon may be heuristically explained by observing that $f(\rho)$
grows {\it only linearly} near $\rho=0,$ forcing the gradient of a minimizer
to be ``smaller than usual'' wherever it is possible. Also, a simple
inspection informs us that the operator in \eqref{divFormForU} behaves like the
$1$-laplacian near critical points and the ordinary laplacian for large
values of $|\na u|.$ This set of remarks make us claim that
equations \eqref{divFormForU} and \eqref{nondivu} are ``more degenerate''
than the $p$-Laplace equation for $1<p<\infty$ but
``less degenerate'' than the $1$-Laplace equation, and
for this reason they deserve attention.
\par
Let us finally observe that functionals and equations with a structure
similar to the one described in this subsection have been
considered in the study of torsional creep problems
in {\it elasto-plastic} materials (\cite{Kawohl}, \cite{langenbach}, \cite{philippin},
\cite{paynePhilippin}). 
\par 
The aforementioned reasons motivate our interest on a
more detailed analysis of such functionals and equations.

\subsection{Main results.}
We shall consider a {\it strictly convex} functional of type \eqref{defJ}.
From now on, unless differently specified, $\Om$ will be a {\it bounded}
domain in $\RN$, while $\f$ is assumed to abide to the requirements
below:
\begin{subequations}\label{eq:cond:f}
\begin{equation}
  \label{eq:cond:fConvex}
  \f \text{ is strictly convex};
\end{equation}
\begin{equation}
  \label{eq:cond:fRegularity}
  \f\in C\sp{1}\big(\Rhc\big)\cap C\sp{2,\lambda}\sb{loc}\big(\Rho\big), %
                      \ 0<\lambda<1;
\end{equation}
\begin{equation}
  \label{eq:cond:fPrime}
  \f'(0)>0.
\end{equation}
\end{subequations}

In order to avoid technicalities, unnecessary to the aims of our
investigation, we require that the couple $(\Om, \psi),$ where $\psi:\pa\Om\to\RR$
is a given continuous function, satisfies a {\it bounded slope condition}
(referred to by \bsc{} from now on; see Section 3 for details). 
\par
Under these assumptions, a classical result makes sure that the variational problem
\begin{equation}
\label{varpb}
\min\{\cJ(w): w\in \mathrm{Lip}(\overline{\Om}),\, w=\psi \mbox{ on } \pa\Om\},
\end{equation}
$\mathrm{Lip}(\overline{\Om})$ being the space of Lipschitz continuous functions in
$\overline{\Omega}$,  
admits a unique solution (see e.g. \cite{giustiDirectMethods}).
\par
In Theorem \ref{mainResultForU}, we
specify sufficient conditions on $f$ 
that guarantee that each solution $u$ of \eqref{varpb} is a viscosity solution of \eqref{nondivu}
subject to $u=\psi$ on $\pa\Om.$ The proof of Theorem
\ref{mainResultForU} follows the outline of the one
given in \cite{magnaniniTalenti2} for the special case
\eqref{deff}: we uniformly approximate $\cJ$ by 
a sequence of strictly convex differentiable functionals 
\begin{equation}
  \label{eq:approxFunctionalForU}
  \cJ\sb{\a}(u)=\ints_\Om f\sb{\a}(|\na u|)\, dx 
\end{equation}
whose minimizers $u\sb{\a}$ are proven to 
be viscosity solutions of some relevant differential equations with
coefficients that converge uniformly 
to those of \eqref{nondivu}. Differently from
\cite{magnaniniTalenti1}, the uniform convergence of the $u\sb{\a}$'s,
needed to use the standard 
stability result of \cite[Section 6]{viscosityUserGuide}, is easily
obtained by means of the \bsc. 
\par
The main result of this paper concerns the set of critical points of a solution of \eqref{varpb}.
\begin{thm}
  \label{minimizer:hasNoIsolatedCriticalPoint}
  Let $u$ be a $C\sp1$ solution of \eqref{varpb}, where $\f$ satysfies
  \eqref{eq:cond:f}.\cri 
  Then $u$ can not have isolated critical points.
\end{thm}
\par
This result considerably improves the one obtained in
\cite{magnaniniTalenti1} for solutions 
of class $C^2$ and settles a conjecture raised by G. Talenti\footnote{Personal communication.}.
Its proof proceeds by contradiction and relies on two remarks: 
\begin{itemize}
\item[(i)] if a solution $u$ of \eqref{varpb} has an isolated critical point at $z_0\in\Om,$
then it is a weak solution of \eqref{divFormForU} in a neighborhood $\cU$ of $z_0;$ 
\item[(ii)] even if $u$ is assumed to be only $C^1$ in $\cU,$ yet one can define
an {\it index} $I(z_0)$ for the vector field $\na u$ in $z_0$ --- a {\it winding number} defined on loops avoiding $z_0.$
\end{itemize}
Remark (i) then implies that $u$ is a classical solution of \eqref{nondivu} in $\cU\setminus\{ z_0\}$
and also that there exists a (distributional) {\it stream function} $v$ for $u$ in $\cU\setminus\{ z_0\},$  
that is a function $v$ such that
\begin{eqnarray}
\label{stream}
\pa_x v=-f'(|\na u|)\,\frac{\pa_y u}{|\na u|}, \ \ \pa_y v=f'(|\na u|)\,\frac{\pa_x u}{|\na u|}
\end{eqnarray}
in $\cU\setminus\{ z_0\}$, in a distributional sense (stream functions
will play a crucial r\^ole in the following sections. The reader may
refer to \cite{aronssonStream} as a propaedeutic reading for what
concerns Section \ref{approx}, while \cite{alessandriniMagnanini2} show how
stream functions have been sometimes used to infer critical-point set
properties). The modulus $\gm{v}$ of 
the gradient of $v$ is proven to 
extend continuously to $z_0$ and, since both $\na u$ and $\na v$ must have the
same index, we infer that $\ii(z_0)=0,$ which entails a contradiction
versus the hypothesis of $z_0$ being isolated.
It is clear that a possible generalization of Theorem
\ref{minimizer:hasNoIsolatedCriticalPoint} to general dimension  should rely on different arguments.
\par
The crucial role played by the stream function $v$ in the proof of
Theorem \ref{minimizer:hasNoIsolatedCriticalPoint} motivates 
a better understanding of system \eqref{stream}
or its inverse
\begin{equation}
\label{invstream}
\pa_x u=g'(|\na v|)\,\frac{\pa_y v}{|\na v|}, \ \ \pa_y u=-g'(|\na v|)\,\frac{\pa_x v}{|\na v|},
\end{equation}
which can be also viewed as sorts of Cauchy-Riemann
systems for $u$ and $v.$  Here $g$ is the {\it Fenchel conjugate of} $f$ defined by 
\begin{equation}
  \label{eq:fenchel}
  g(r)=\sup\{\rho r-f(\rho): \rho\ge 0\}, \quad r\in [0,\infty).
\end{equation}
In other words, we want to investigate on
the possibility of defining a generalized {\it stream function} $v$ associated
to a solution $u$ of \eqref{varpb}. 
The main difficulty with this task is that,
since by Theorem \ref{minimizer:hasNoIsolatedCriticalPoint} $u$ may not
have isolated critical points, system \eqref{stream}
is in general severely singular. 
\par
In this paper, we do not completely secceed in our task, but
we present a few results which may help to understand the problem.
\par
In Section \ref{approx}, we show that, owing
to the properties of the chosen lagrangeans $f\sb{\a},$  
the system
\begin{equation}
\label{streamm}
\pa_x v\sb{\a}=-f'\sb{\a}(|\na u\sb{\a}|)\,\frac{\pa_y u\sb{\a}}{|\na u\sb{\a}|}, \ \ \pa_y v\sb{\a}=f\sb{\a}'(|\na u\sb{\a}|)\,\frac{\pa_x u\sb{\a}}{|\na u\sb{\a}|}
\end{equation}
can be uniquely solved by suitably normalized stream functions $v\sb{\a},$
which are critical points of functionals,
\begin{equation*}
  \label{eq:approxForV}
  \cK\sb{\a}(v)=\ints_\Om g\sb{\a}(|\na v|)\, dx,
\end{equation*}
where $g'\sb{\a}=(f'\sb{\a})^{-1}$ ($g\sb{\a}$ is indeed the Fenchel
  conjugate of $f\sb{\a}$ defined accordingly to \eqref{eq:fenchel}).
It is evident that $\ua$ and $\va$ also satisfy
\begin{equation}
\label{invstreamm}
\pa_x u\sb{\a}=g'\sb{\a}(|\na v\sb{\a}|)\,\frac{\pa_y v\sb{\a}}{|\na v\sb{\a}|}, \ \ \pa_y u\sb{\a}=
-g'\sb{\a}(|\na v\sb{\a}|)\,\frac{\pa_x v\sb{\a}}{|\na v\sb{\a}|}.
\end{equation}
In Theorem \ref{mainResultForSys}, under appropriate
assumptions on the approximating sequence $(\fa)_{\a\in\NN},$ we show that the sequences $(\ua)_{\a\in\NN}$ and 
$(\va)_{\a\in\NN}$ contain subsequences which converge respectively to functions $u$ and $v$
satisfying \eqref{invstream} almost everywhere. 
\par
We are not able to prove that $u$ and $v$ also satisfy \eqref{stream}; however,
by the same argument used in the proof of 
Theorem \ref{mainResultForU}, we show that $v$ is a viscosity solution $v$ of 
\begin{equation}
\label{nondivv}
\tr [\cB(\na v)\,\na^2 v]=0,
\end{equation}
where $\cB(p)$ is a matrix whose coefficients are the uniform
limits of 
$$
-\big[1-\aaa(\ga'(|p|))\big]p_ip_j - |p|^2\aaa(\ga'(|p|))\de_{ij}
\quad i, j=1,\dots, 2,
$$
with
\begin{equation}
\label{aalfa}
 \aaa(\rho) = %
  \begin{cases}
    \dfrac{\rho \fa''(\rho)}{\fa'(\rho)}, & \rho>0, \\
    \lim\limits_{\rho\to 0^+} \dfrac{\rho \fa''(\rho)}{\fa'(\rho)},& \rho=0.
  \end{cases}
\end{equation}
This is the content of Theorem \ref{mainResultForV}.
\par
It is worth mentioning that the analytic form of $\cB$ may depend upon
the particular approximating sequence $(\fa)\sb{\adom}$ adopted,
that is, different approximations lead to different limit equations (see
Remark \ref{practicalApprox} for details). One of these choices leads to the following
interesting equation for $v$:
\begin{equation*}
  -\big[1-\aaalim(\g'(\gm{v}))\big]\lp\sb\infty v - %
  \gm{v}^2\aaalim(\g'(\gm{v}))\lp v = 0,
\end{equation*}
where $g$ is given by \eqref{eq:fenchel}. Notice
that, for values of $\gm{v}$ less than or equal to $f'(0),$ $v$
must be $\infty$-harmonic.

\section{Critical points of minimizers}
\label{sec:crit}\label{criticalPoints}
\label{criticalPoints:distributionalEquation}
This section will be devoted to the proof of our main result (Theorem
\ref{minimizer:hasNoIsolatedCriticalPoint}), which will be consequence
of Lemma \ref{bridgeToMinimizers} and Theorem
\ref{distributionalSolution:hasNoIsolatedCriticalPoint},
which may be of independent interest.

\begin{lemma}
  \label{bridgeToMinimizers}
  Let $u$ be a solution of \eqref{varpb} with $\f$ satisfying
  \eqref{eq:cond:f}. If $u\in C\sp{1}(\Omega)$ and the set
  $\{z\in\Omega: \gm{u(z)}=0\}$ has zero Lebesgue measure, then $u$ is
  a weak solution of \eqref{divFormForU} in $\Omega$.
\par
In particular, if $\z\in\Om$ is an isolated critical point for $u,$ then 
there exists a neighborhood of $\z$ in which $u$ is a weak solution of \eqref{divFormForU}.
\end{lemma}
\begin{proof}
  It is easy to see that, for any
  test function $\phi$ whose support is contained in $\Om$,
  the derivative of $\jj$ in the direction given by $\phi$ may be
  written as
  \begin{displaymath}
    \partial\jj(u)(\phi)=\!\!\!\!\!%
       \int\sb{\{\gm{u}\neq0\}}%
       \!\!\!\!\!\f'(\gm{u})\,
                   \langle\dfrac{\diff u}{\gm{u}},\diff\phi\rangle\, dxdy + 
                   \f'(0)\!\!\!\!\int\sb{\{\gm{u}=0\}}%
       \!\!\!\gm{\phi}\, dxdy.
  \end{displaymath}
  By assumption, the second addendum vanishes, while the first one
  amounts to
  \begin{displaymath}
    \int\sb{\Om}\f'(\gm{u})\,
                  \langle\dfrac{\diff u}{\gm{u}},\diff\phi\rangle\, dxdy.
  \end{displaymath}
  If $u$ is a solution of \eqref{varpb}, then $\partial\jj(u)(\phi)=0$ for every test function
compactly supported in $\Om$ and hence $u$ is a weak solution of \eqref{divFormForU} in $\Om.$
\end{proof}

We will next proceed to compute the index $\ii(\z)$ of an isolated
critical point of a solution $u\in C\sp{1}(\Omega)$ of \eqref{varpb}.
We recall that $\ii(\z)$ is defined by the
  formula
  \begin{equation}
    \label{def:index}
  \ii(\z) = \dfrac{1}{2\pi}\int\sb{+\gamma}%
                              \dfrac{u\sb{x}du\sb{y}-u\sb{y}du\sb{x}}%
                                    {|\diff u|^2},
  \end{equation}
  where $+\gamma$ is any loop which wraps $\z$ counterclockwise and no
  other critical point. \cri
Also recall the geometric meaning of the previous definition: the index of
a critical point of a $C\sp1$-regular 
function is defined as the topological index of the vector field $\na u/|\na u|$ at the same
point and that the latter corresponds to the topological degree of the field $\na u/|\na u|$
itself, considered as a map of the unit circle in itself.

Given any $r>0$ such that $B = \overline{B(z_0,r)}\subset\Omega$
we set our first goal to proving that the differential form
\begin{equation}
  \label{eq:4-omegaBis}
  \omega = \dfrac{\fb(\gm{u})}{\gm{u}}\,(-u\sb ydx
  + u\sb xdy),
\end{equation}
which is continuous and bounded in $\bp = \b\setminus\{\z\}$, may
be integrated to obtain a so-called \emph{stream function} (see
\cite{aronssonOnStreamFunctions}) $v$ which 
is continuous in $B$. 

Notice that \eqref{divFormForU} may be cast into the form
\begin{equation}
  \label{eq:currentFormForU}
  d\omega = 0,\qquad \text{in }\Omega,
\end{equation}
in the sense of currents/distributions, where $d$ has to be
interpreted as the boundary operator (see \cite{Simon}).
\begin{lemma}
\label{v:existence}
Let $u:\Omega\longrightarrow\R$ be a $C\sp1(\Omega)$ distributional
solution to \eqref{divFormForU}, and let $\z$ be an isolated
critical point for $u$. Then the following claims hold.
\begin{enumerate}[(i) ]
\item\label{inner:second:first}
  the period of the 1-form $\omega$ given in \eqref{eq:4-omegaBis}
  around $\z$ is null;
\item\label{inner:second:second} there exists a function
  $v\in C\sp1(\bp)\cap C\sp0(\b)$ such that
  $dv=\omega$;
\item moreover, $\gm{v}\in C\sp0(\b)$.
\end{enumerate}
\end{lemma}
\begin{proof}
(i) Let $B_1=\ball{\z}{r/2}$. We must prove that $\int\sb{\partial
  B_1}\omega=0$. 
    Indeed, we know that, for any $\varphi\in C\sp{\infty}\sb0(\b)$
    it holds:
    \begin{displaymath}
      \int\sb{\b}\omega\wedge d\phi = 0.
    \end{displaymath}
    (This is really what \eqref{eq:currentFormForU} means.)
    Let $(\eta\sb\varepsilon)\sb{\varepsilon>0}$ be a family of
    radially symmetric
    regularizing kernels, and define $\phi\sb\varepsilon =
    \eta\sb\varepsilon\ast\chi\sb{B_1}$, where as usual $\chi\sb{B_1}$ is the
    characteristic function of the set $B_1$. Then $d\chi\sb{B_1} =
    \mathcal{H}\sp1\llcorner\partial B_1(-\nu\sb1 dx -\nu\sb2 dy)$
    ($\nu=(\nu\sb1,\nu\sb2)$ is the outer unit normal to the domain
    $B_1$, while $\llcorner$ denotes the operator of restriction of
    measures to subsets) and
    $\spt(d\phi\sb\varepsilon)$ is contained in a tubular $\varepsilon$-neighborhood
of $\partial B_1$, hence
    taking $\varepsilon\sb0$ small enough,
    $\z\notin\spt(d\phi\sb\varepsilon)$ and
    $\spt(d\phi\sb{\varepsilon})\subset B$ 
    for all $\varepsilon\le\varepsilon\sb0$ and it holds (since
    $\spt(d\phi\sb{\varepsilon})\subset\spt(d\phi\sb{\varepsilon\sb0})$)
    \begin{displaymath}
      0 = \int\sb{\b}\omega\wedge d\phi\sb\varepsilon =
          \int\sb{\spt(d\phi\sb{\varepsilon\sb0})}\!\!\!\!\!\!\omega%
                                                   \wedge d\phi\sb\varepsilon.
    \end{displaymath}
    The last integral tends to 
    \begin{displaymath}
          \int\sb{\spt(d\phi\sb{\varepsilon\sb0})}\!\!\!\!\!\!\omega%
                                                   \wedge d\chi\sb{B_1}=
          \int\sb{\partial B_1}\omega.
    \end{displaymath}
    The alleged convergence is worth an explanation: we know (\eg see
    ~\cite[Thm. 2.2]{AFP}) that both $d\phi\sb\varepsilon$ tends to
    $d\chi\sb{\b}$ and the total variation $|d\phi\sb\varepsilon|$
    tends to $|d\chi\sb{\b}|$ in the sense of measures. This is enough
    to prove convergence in the stronger topology dual to the space of
    continuous
    and bounded functions on $\spt(d\phi\sb{\varepsilon\sb0})$ (see
    ~\cite[Prop. 2, pg. 38]{giaquintaModica}), to which $\omega$
    belongs.
    \renewcommand{\b}{B}

\par
(ii) Thus we can integrate $\omega$, to obtain a function $v\in
    C\sp1(\bp)$. We claim that $v$ can be extended
    continuously to $\b$. In fact,
    since $dv=\omega$, then \eqref{stream} holds and hence $\diff v\in L\sp\infty(\bp)$;
    therefore $v$ is (Lipschitz and \af) uniformly continuous on
    $\bp$ and then we can extend it continuously to the border of
    $\bp$, in particular to $\z$. This allows us to \emph{mend} the
    domain of definition of $v$ from the topological point of view.
\par
(iii) From the definition of $\omega$, it turns out that
$|\diff v| = \fb({\gm{u}})$; since $u\in C\sp1(\Omega)$ and $\fb$ is
continuous, $|\diff v|$ is continuous on $\b$ too.
\end{proof}

Now we prove that the index of $\z$ as a critical point of the function
$u$ is zero. We recall the
following proposition (see \eg ~\cite[Lemma 3.1]{alessandriniMagnanini}).
\begin{externalTheorem}
\label{th:ex:alessandriniMagnanini}
Let $w$ be a real-valued $C\sp1$ function in an open set $\Omega$ in
the complex plane. Let $z\sb0\in\Omega$ be an isolated critical point
of $w$.\cri
Then, one of the following cases occurs.
\begin{enumerate}[(i) ]
\item There exists a neighborhood $\mathcal{U}$ of $z\sb0$ such that
  $\{z\in\mathcal{U}: w(z)=w(z\sb0)\}$ is exactly $z\sb0$, and we have
  $\ii(z\sb0) = 1$.
\item There exists a positive integer $L$ and a neighborhood
  $\mathcal{V}$ of $z\sb0$ such that the level set
  $\{z\in\mathcal{V}: w(z)=w(z\sb0)\}$ consists of $L$ simple
  curves. If $L\ge2$, each pair of such curves crosses at $z\sb0$
  only. We have $\ii(z\sb0) = 1 - L$.
\end{enumerate}
\end{externalTheorem}

The previous theorem let us prove the following crucial statement.
\begin{thm}\label{v:c1}
  Let $u$ be a $C\sp{1}$ solution of 
  \eqref{divFormForU} in the sense of
  distributions. Assume that $\z$ is an isolated critical point for
  $u$.\cri
  Then the index of $\z$ as a critical point is null: $\ii(\z)=0$.
\end{thm}
\begin{proof}
In the following, we will need to recall that the level lines of
$u$ correspond to the lines of steepest descent for $v$.

We apply Theorem \ref{th:ex:alessandriniMagnanini} to the function
$u$, to infer the geometry
of its level set $\{u=u(\z)\}$. We can exclude the case
$\{u=u(\z)\}=\{\z\}$, for otherwise $\z$ would be a local extremum for
$u$ and hence for some $\varepsilon>0)$ either $\{u = u(\z) +
\varepsilon\}$ or $\{u = u(\z) - \varepsilon\}$ would be a closed
curve $\gamma$ winding around $\z$ thus implying that
\begin{displaymath}
  \int\sb{+\gamma}\omega\neq 0.
\end{displaymath}
But this would contradict Lemma \ref{v:existence},
\eqref{inner:second:first}.

\begin{figure}[t]
\begin{minipage}[t]{0.3\linewidth}
  \parbox{0.3\linewidth}{\hbox{}}
\end{minipage}%
\begin{minipage}[t]{0.5\linewidth}
\resizebox{5cm}{!}{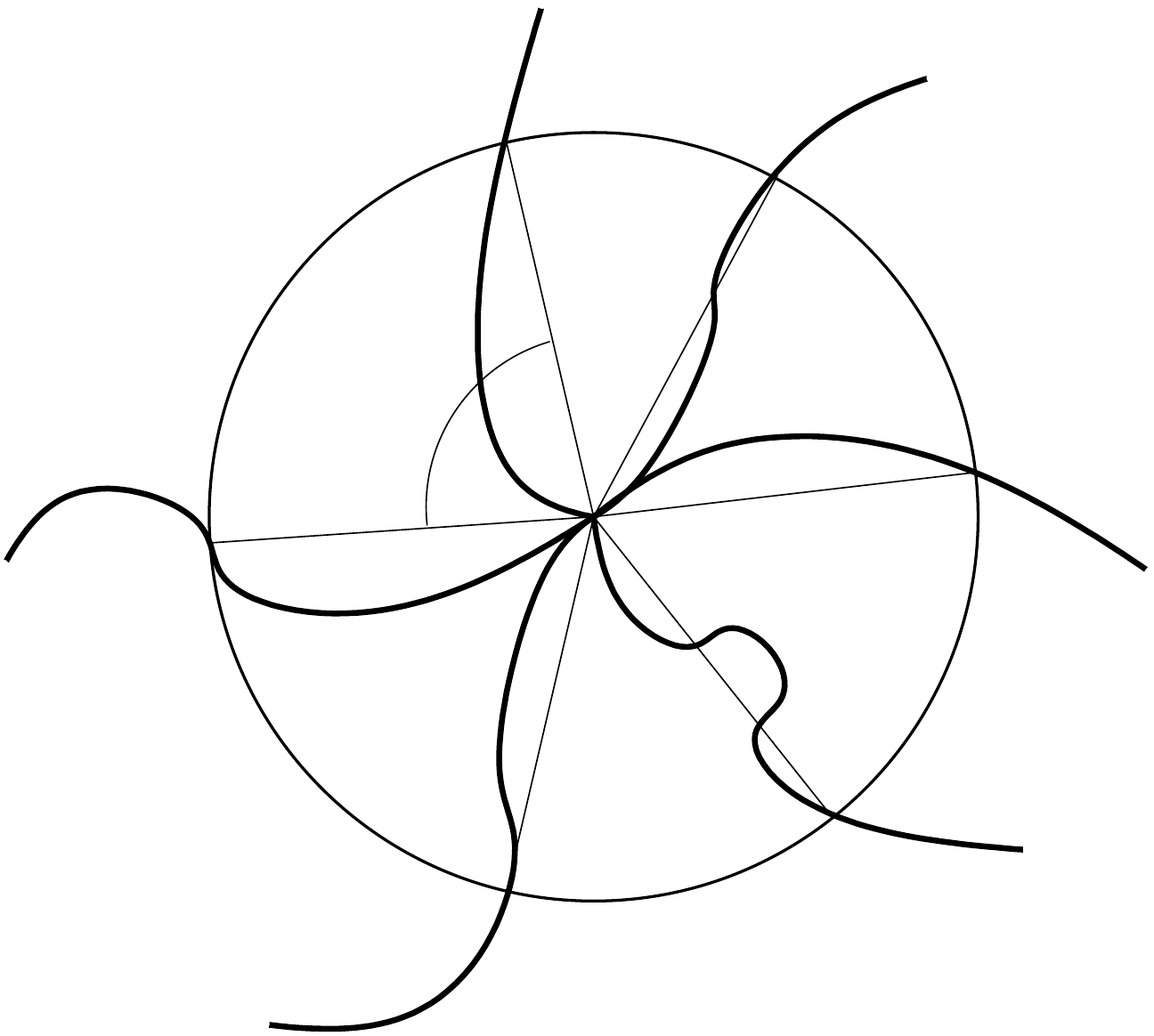}
\caption{The level set $\{u=u(\z)\}$.}
\label{fig:CP-uLevelSet1}
\end{minipage}
\end{figure}

Let $2L$ be the number of branches of $\{u=u(\z)\}$. (\caveat: in the
present terminology, a branch is any arc in which any of the curves whose
existence is stated by Theorem \ref{th:ex:alessandriniMagnanini} is
split into by $\z$. Hence, when the curves are $L$, the branches are
exactly $2L$.) Denote them by $\gamma\sb{1}\ldots\gamma\sb{2L}$; here
the subscripts are assigned in the counterclockwise order of
occurrence, starting from an arbitrary branch.
\cri
We preliminarly observe that each $\gamma\sb{i}$ is rectifiable. Indeed, let 
$z, z'\in\gamma\sb{i}\cap B';$ since $\gamma\sb{i}$ is a curve of steepest descent for $v$ and
$\displaystyle\gm{v(\z)}\le\gm{v}$ on $B,$
we can write that
\begin{displaymath}
  \begin{split}
    \displaystyle\max\sb{\b}|v-v(z')| \ge %
    v(z)-v(z')=%
	\int\sb{0}\sp{l}|\diff v(\gamma\sb{i}(s))|\, ds\ge%
    l\,\,|\diff v(\z)|,
  \end{split}
\end{displaymath}
where $l$ is the length of the arc on $\gamma\sb{i}$ joining $z'$ to $z.$
Thus, we discover that $l$ remains bounded as $z'\to\z,$ since $|\diff v(\z)|>0.$
\cri
When $r$ is small enough, each $\gamma\sb{i}$ crosses the circle
$\{z:|z-\z|=r\}$ in one point $z\sb{i}(r)$, for
$i=1\ldots2L$.
Hence, setting $\theta_i(r)$ as the angle between
the two directions $z\sb{i}(r)-\z$ and $z\sb{i+1}(r)-\z$ for
$i=1\ldots2L$ and $z\sb{2L+1}(r)=z\sb{1}(r)$, it holds that
\begin{equation}
  \label{eq:sumOfAlphas}
  \sum\sb{i=1}\sp{2L}\theta_i(r)=2\pi.
\end{equation}
\cri
Now consider two consecutive branches, say $\gamma\sb{i}$,
$\gamma\sb{i+1}$; we may assume that
$v$ is increasing along
$\gamma\sb{i}$ and decreasing along $\gamma\sb{i+1}$, away from $\z$. (It is
easy to see that the case in which $v$ is increasing --- or decreasing --- along
two consecutive branches is not consistent with the
present case.) 
Since the two branches are
rectifiable, 
then we can infer that
\begin{align*}
  v(z\sb{i})=&\ v(\z) + |\diff v(\z)|\,\, r + o(r), \\
  v(z\sb{i+1})=&\ v(\z) - |\diff v(\z)|\,\, r + o(r),
\end{align*}
as $r$ approaches $0,$ where we have set for short $z\sb{i}=z\sb{i}(r),\, z\sb{i+1}=z\sb{i+1}(r).$
\cri
At the same time, there exists $\xi\in[z\sb{i}, z\sb{i+1}]$ (here $[z\sb{i},
z\sb{i+1}]$ is the line segment joining $z\sb{i}$ to $z\sb{i+1}$) such
that the following inequality holds
\begin{equation*}
  \label{ineq:xi}
  v(z\sb{i})-v(z\sb{i+1})\le|\diff v(\xi)|\cdot|z\sb{i}-z\sb{i+1}|.
\end{equation*}
Thus, we conclude
\begin{displaymath}
  |\diff v(\xi)| \ge \dfrac{v(z\sb{i})-v(z\sb{i+1})}{|z\sb{i}-z\sb{i+1}|} \ge %
        \dfrac{2\cdot|\diff v(\z)| + o(1)}%
              {\sqrt{2\big[1-\cos\,\theta_i\big]}}.
\end{displaymath}
Let $\Theta$ be any
limit point of $\theta_i=\theta_i(r)$, as $r$ tends to $0$.
Then we obtain (eventually by taking subsequences) that
\begin{displaymath}
  \gm{v(\z)}=\liminf\sb{r\to0}|\diff v(\xi)| \ge
  \dfrac{2}{\sqrt{2(1-\cos\Theta)}}\,|\diff v(\z)|\ge%
  \gm{v(\z)}.
\end{displaymath}
Therefore $\Theta=\pi$.

Thus we have proved that two branches of $\{u=u(\z)\}$ cannot exist
such that the angle formed by the limit tangent versor to the branches
is different from $\pi$. But then \eqref{eq:sumOfAlphas} informs that
the level set $\{u=u(\z)\}$ is made of no more than two branches, \ie\
one curve. Hence $L=1$ and then, as stated by Theorem
\ref{th:ex:alessandriniMagnanini} the index of $\z$ as a critical point 
is zero.
\end{proof}

\begin{lemma}
  \label{indexFormula}
  Let $w$ be any $C\sp{2}$ function and let $t>0$ be any regular value
  of $\gm{w}$. \cri
  Then
  \begin{equation*}
    \int\sb{+\gamma\sb{t}}\dfrac{w\sb{x} dw\sb{y} - w\sb{y}
      dw\sb{x}}{\gm{w}^2} = %
    \int\sb{+\gamma\sb{t}}\dfrac{\det(\Hess w)}{|\Hess w\diff w|}\, ds,
  \end{equation*}
  where $\gamma\sb{t}=\{z: w(z)=t\}$ and $s$ denotes the arc-length.
\end{lemma}
\begin{proof}
  We have
  \begin{equation*}
    \int\sb{+\gamma\sb{t}}\dfrac{w\sb{x} dw\sb{y} - w\sb{y}
      dw\sb{x}}{\gm{w}^2} = %
    \int\sb{+\gamma\sb{t}}\dfrac{(w\sb{x} w\sb{xy} - %
      w\sb{y} w\sb{xx})dx + %
      (w\sb{x} w\sb{yy} - w\sb{y} w\sb{xy})dy
    }{\gm{w}^2}.
  \end{equation*}
  Since $t$ is a regular value, $\gamma\sb{t}$ is made of regular
  curves and
  \begin{equation*}
    \left(\begin{matrix}
      dx \\
      dy
    \end{matrix}\right) = %
    \dfrac{ds}{|\Hess w\diff w|}%
    \left(\begin{matrix}
            -w\sb{x}w\sb{xy} -w\sb{y}w\sb{yy} \\
  \phantom{-}w\sb{x}w\sb{xx} +w\sb{y}w\sb{xy}
          \end{matrix}\right)
  \end{equation*}
  on $\gamma\sb{t}$. The conclusion follows at once after simple
  algebraic manipulations.
\end{proof}

\begin{thm}\label{distributionalSolution:hasNoIsolatedCriticalPoint}
  Let $u$ be a $C\sp1$ solution in the sense of distributions of
  \eqref{divFormForU}. Assume $f$ satisfies
  \eqref{eq:cond:f}.\cri
  Then $u$ can not have isolated critical points.
\end{thm}
\begin{proof}
  Let $\Omega'$ denote the open subset of $\Omega$ in which
  $\gm{u}>0$. On any open subset $A$ of $\Omega'$ whose closure is
  contained in $\Omega'$, $\gm{u}$ is bounded away from zero; thus we
  can apply ~\cite[Thm. 10.18]{giustiDirectMethods} to infer that $u$ has
  H\"older continuous second derivatives in any such $A$ and hence in
  $\Omega'$ (notice that, under our assumptions \eqref{eq:cond:f} on $\f$,
  it is a standard computation to prove that
  $u\in\wss(\Omega')$). Also, on any such $A$, Sard's lemma in the
  version of 
  \cite{korobkov} may be applied to $\gm{u}^2$ in $A$; in particular, we have
  that
  \begin{equation}
    \label{eq:coarea}
    \int\sb{A}\dfrac{\det(\Hu)}{\gm{u}}dxdy=\int\sb{m}\sp{M}\bigg(\int\sb{\gamma\sb{t}}
    \dfrac{\det(\Hu)}{|\Hu\diff u|}\bigg)dt
  \end{equation}
  by the coarea formula (see ~\cite[Thm. 1,
  Sec. 3.4.2]{evansGariepy}), where
  $m=\displaystyle\min_{\overline{A}}\gm{u}$,
  $M=\displaystyle\max_{\overline{A}}\gm{u}$.\cri
  Now $\z$ is a strict minimum point for $\gm{u}$ such that the
  connected component $\tilde{A}$ of
  $\{z\in\Omega':\gm{u(z)}<\varepsilon\sb{0}\}$ containing $\z$ is
  bounded by a simple closed curve. We now choose $A\sb{\varepsilon} =
  \{z\in\tilde{A}: \varepsilon<\gm{u(z)}<\varepsilon\sb{0}\}$ for
  $0<\varepsilon<\varepsilon\sb{0}$ and apply \eqref{eq:coarea} and
  Lemma \ref{indexFormula}; we get
  \begin{equation*}
    \int\sb{A\sb{\varepsilon}}\dfrac{\det(\Hu)}{\gm{u}}\, dxdy = %
    \int\sb{\varepsilon}\sp{\varepsilon\sb{0}}%
      \bigg(%
        \int\sb{\gamma\sb{t}}\dfrac{u\sb{x}du\sb{y}-u\sb{y}du\sb{x}}{\gm{u}^2}%
      \bigg)\, dt.
  \end{equation*}
  Formula \eqref{def:index} and Theorem \ref{v:c1} yield that the integrand on the
  right-hand side of the latter is null for almost every
  $t\in(\varepsilon,\varepsilon\sb{0})$, hence 
  \begin{equation*}
    \int\sb{A\sb{\varepsilon}}\dfrac{\det(\Hu)}{\gm{u}}dxdy = 0.
  \end{equation*}
  Now, $\det(\Hu)\le0$ in $A\sb{\varepsilon}$, because $u$ satysfies
  an elliptic equation; thus $\det(\Hu)\equiv0$ in $A\sb{\varepsilon}$
  and hence in $A\sb{0}$, since $\varepsilon\in(0,\varepsilon\sb{0})$
  is arbitrary. An application of Bernstein's inequality (\eg see
  ~\cite[Problem 12.3]{gilbargTrudinger} or
  ~\cite{talentiOnBernstein}) yields that
  \begin{equation*}
   |\na^2 u|^2 \le 
     -\det(\na^2 u)\,\big[\aaalim(\gm{u}) + \aaalim(\gm{u})^{-1}\big] 
  \end{equation*}
  and hence that $u$ is affine in $A\sb{0}$ and, by continuity,
  constant in a whole neighborhood of $\z$ (recall that
  $\gm{u(\z)}=0$). This is a contradiction.
\end{proof}

\begin{proof}[Proof of Theorem \ref{minimizer:hasNoIsolatedCriticalPoint}]
  By Lemma \ref{bridgeToMinimizers} $u$ satisfies \eqref{divFormForU}
  in the sense of distributions. Therefore Theorem
  \ref{distributionalSolution:hasNoIsolatedCriticalPoint} applies.
\end{proof}
\begin{rmk}
  It is easy to see that a result similar to that of Theorem \ref{minimizer:hasNoIsolatedCriticalPoint}
can be proved also for $C^1$ solutions of \eqref{stream}: if $(u,v)$ is a 
$C^1$ pair satisfying \eqref{stream}, then $u$ cannot have isoltated critical points. Indeed,
  in this case Theorem \ref{v:c1} holds for $v$ taken as the second component of the 
solution pair.
\end{rmk}
 \begin{rmk}
   A slight modification of the example in \cite[Sec. 2.2]{magnaniniTalenti1}
   proves that there exist solutions of \eqref{varpb} whose set of
   critical points has zero Lebesgue measure. Indeed the function $u$
   defined by
   \begin{equation*}
     u(x,y)= %
     \left\{%
       \begin{aligned}
         1-\sqrt{\dfrac{1-x^2-y^2}{2} %
                 +\dfrac{1}{2}\sqrt{(1-x^2-y^2)^2+4y^2}%
           } &\qquad x\ge0, \\
         -1+\sqrt{\dfrac{1-x^2-y^2}{2} %
                 +\dfrac{1}{2}\sqrt{(1-x^2-y^2)^2+4y^2}%
           } &\qquad x<0, \\
       \end{aligned}%
     \right.
   \end{equation*}
   is a distributional solution of class $C\sp{1,1}$ of \eqref{divFormForU} in
   every ball $\ball{0}{R}$, $0<R<1$, and hence solves \eqref{varpb}
   with $\psi=u\sb{|\sb{\partial\ball{0}{R}}}$.
 \end{rmk}

\section{Stream Functions and Viscosity Approximations}
\label{approx}

The reader should refer to \cite{viscosityUserGuide} for a definition and
relevant developments concerning
viscosity solutions (\cite{koike}
for an account more appropriate to novices). Lemma
\ref{correctedViscousFormForGeneralWithModulus} resumes the 
only non-standard property of viscosity solutions needed in the
following.

Here, we shall analyse the relationships between functional
\eqref{defJ} and equation \eqref{nondivu}; we will also set up a
framework that, in certain instances, leads to the construction of a
stream function associated to the unique solution of \eqref{varpb}.

We shall assume that $f$ fulfills \eqref{eq:cond:f} and that the ratio
$\dfrac{\rho\f''(\rho)}{\f'(\rho)}$ has finite limits as
$\rho\to0\sp+$ and $\rho\to+\infty$. It is easy to show that on
account of \eqref{eq:cond:fPrime}, the first limit is zero. Thus, we
adopt the definition \eqref{alphaIntro} of the function
$\fun{\aaalim}{\Rhc}{\Rhc}$ and we assume
\begin{equation}
  \label{def:alphaToInfty}
  \aaainfty:=\lim\sb{\rho\to+\infty}\dfrac{\rho\f''(\rho)}{\f'(\rho)} >0.
\end{equation}

Observe incidentally that \eqref{def:alphaToInfty} and \eqref{eq:cond:fRegularity}
imply that there exist constants $q\sb1, q\sb2\in(1,+\infty),
q\sb1<q\sb2$ and $c\sb1, c\sb2>0$ such that, for $\rho\ge0$
\begin{align*}
  c\sb1\rho^{q\sb1}-c\sb2\le&\f(\rho)\le c\sb1\rho^{q\sb2} + c\sb2, \\
  c\sb1\rho^{q\sb1-1}-c\sb2\le&\f'(\rho)\le c\sb1\rho^{q\sb2-1} + c\sb2.
\end{align*}

As starters, we recall a classical result, providing a short proof of
it tailored on our purposes.
We will say that a couple $(\Omega, \psi),$ where $\psi:\pa\Omega\to\RR$ is
a continuous function, satisifies a {\it bounded slope condition} (\bsc{} for short) with constant $Q>0$ 
if, for every $x_0\in\pa\Om$ 
there exist two {\it affine functions} $L^+$ and $L^-$ such that
\begin{equation}
  \label{bsc}  
  \begin{aligned}
    &L^-\le\psi\le L^+ \ \mbox{ in } \ \pa\Om,\\
    &L^-(x_0)=\psi(x_0)=L^+(x_0),\\
    &\sup\limits_{x,y\in\Om\atop x\not= y}\frac{|L^\pm(x)-L^\pm(y)|}{|x-y|}\le Q.
  \end{aligned}
\end{equation}
\begin{prop}
  \label{basicFactsAboutApproxForU}
Let $(\Omega, \psi)$ satisfy a \bsc{} with constant $Q.$ Assume $(\fa)\sb{\adom}$ is a sequence of strictly convex functions
converging uniformly to $f$ on $\Rhc$.  Let $u$
(resp. $\ua$) be the unique solution of \eqref{varpb} for $\jj$
(resp. for $\ja$ in \eqref{eq:approxFunctionalForU}). \cri
Then 
 \begin{enumerate}[(a) ]
  \item\label{cond:minimizing} $\ua$ is a minimizing sequence
    for $\jj$ and $\ja(\ua)\to\jj(u)$;
  \item\label{uApproxApproxU} the sequence
    $(\ua)\sb{\adom}$ tends to $u$ in the \emph{sup norm} topology,
    and in the weak* topology of $\wi$.
  \end{enumerate}
\end{prop}
\begin{proof}
  Since $\ja\to\jj$ uniformly \eqref{cond:minimizing} is standard.
  To prove \eqref{uApproxApproxU} apply ~\cite[Thm
  1.2]{giustiDirectMethods} to get the bound 
\begin{equation}
\label{BdGrads}
\gm{\ua}\le Q \quad \mbox{ on } \ \overline{\Omega}.
\end{equation} 
Then \cite[Ch. 2,
    B.1]{evansWeakConvergence} and an 
    application of the Ascoli-Arzel\`a's theorem yields the desired
    conclusion.
\end{proof}

In order to prove our next proposition, we need a preliminary lemma. 
We state it in quite a general form, since it will also be used further
along, while considering the case of stream functions. The proof of
Lemma \ref{correctedViscousFormForGeneralWithModulus} follows the
lines of that of \cite[Prop. 2.3]{ishibashiKoike} (which may also be considered as
an appropriate source for a more advanced discussion of the 
 current topic)

\newcommand{\np}{\nabla\sb{p}}
\begin{lemma}\label{correctedViscousFormForGeneralWithModulus}
  Let $\fun{M}{\Rn}{\Rn}$ be a monotone operator such that the mapping
  $p\mapsto\dfrac{|p|^3}{|M(p)|}\np M(p)$ is continuous on $\Rn$. \cri
  Then every weak solution of
  \begin{equation}
    \label{eq:generalEulerLagrangeWithModulus}
    \dv\big(M(\diff u)\big)=0
  \end{equation}
  which is continuous and lies in $\wq$ for some $q>1$ is also a
  viscosity solution of 
  \begin{equation}
    \label{eq:inftyLaplacianGeneralWithModulus}%
    -\dfrac{\gm{u}^3}{|M(\diff u)|} %
    \tr\big\{\np M(\diff u) \Hu\big\}= 0,
  \end{equation}
  where $\tr$ is the usual trace operator on matrices.
\end{lemma}
\begin{proof}
  As usual for the viscosity setting, the proof splits into two steps:
  first prove that $u$ is a subsolution, then that $u$ is a
  supersolution. The two steps are
  nearly identical, thus we only go through the first one.

  Assume by contradiction, that there exist $\x\in\Omega$ and
  $\phi\in C\sp2(\Omega)$ satisfying
  \begin{center}
    $u(\x) = \phi(\x)$, \qquad%
    $\phi(x) > u(x)$, $x\in\Omega\setminus\{\x\}$,
  \end{center}
    \begin{equation}
      \label{cond:absurdForGeneralWithModulus} 
      -\dfrac{\gm{\phi(\x)}^3}{|M(\diff\phi(\x))|} %
      \tr\big\{\np M(\diff\phi(\x)) \Hess\phi(\x)\big\} > 0.
    \end{equation}
  By our assumptions on $M$ and $\phi$, the last inequality implies
  that $\gm{\phi(\x)}>0$. Hence, by continuity, we can find numbers
  $\theta, r, t>0$ with $\theta, r>t$ such that 
  \begin{center}
    $\overline{\br}\subset\Omega$, \qquad %
    $\displaystyle\min\sb{\bt}(u-\phi)\ge -\theta$, \qquad %
    $\displaystyle\max\sb{\partial\br}(u-\phi)\le -3\theta$, \\
    $-\tr\big\{\np M(\diff\phi(\x)) \Hess\phi(\x)\big\}> 2\theta$, \qquad%
    $-\tr\big\{\np M(\diff\phi) \Hess\phi\big\} > \theta$ on $\br$
  \end{center}
  (here $\br$ denotes the ball of radius $r$ centered at $\x$).

  Given a non-negative $\psi\in\wqcs$, we integrate by parts on $\br$
  the last inequality and obtain
  \begin{equation}
    \label{eq:integratedViscous}
    \begin{split}
      \int\sb{\br}\langle M(\diff\phi(x)), \diff\psi(x)\rangle =%
      \qquad \qquad \quad\\
      \int\sb{\br}-\tr\big(\J M(\diff\phi(x))\Hphi(x)\big)\cdot\psi(x)%
      \ge %
      \int\sb{\br}\theta\cdot\psi(x).
    \end{split}
  \end{equation}

  Now we choose 
  \begin{displaymath}
    \psi(x)=%
      \begin{cases}
        \big(u(x) - \phi(x) + 2\theta\big)\sp+ &\ \text{for any
                                                }x\in\br,\\
        0 &\ \text{elsewhere};
      \end{cases}
  \end{displaymath}
  then $\psi\in\wqcs$ and 
  \begin{displaymath}
    \diff\psi(x)=%
      \begin{cases}
        \diff u(x) - \diff\phi(x) &\ \text{for any %
        }x\in\br\cap\{u-\phi>2\theta\},\\ 
        0 &\ \text{elsewhere}
      \end{cases}
  \end{displaymath}
  (see \eg ~\cite[(3.8) pg. 86]{giustiDirectMethods}). Thus, by
  \eqref{eq:integratedViscous}, we have 
  \begin{displaymath}
    \label{eq:ansatzIntegratedViscous}
    \begin{split}
      \int\sb{\br\cap\{u-\phi>2\theta\}}%
      \!\!\!\!\!\!\!\!\!\!\!\!%
      \langle M(\diff\phi),\diff u-\diff\phi\rangle \ge %
      \!\!\!\!\!\!\!\!\!\!\!\!%
      \int\sb{\br\cap\{u-\phi>2\theta\}}%
      \!\!\!\!\!\!\!\!\!\!\!\!%
      \theta(u-\phi+2\theta) \ge %
      \int\sb{\bt}\theta^2,
    \end{split}
  \end{displaymath}
  while, since $u$ is a weak solution of
  \eqref{eq:generalEulerLagrangeWithModulus}, it holds that
  \begin{displaymath}
    \label{eq:ansatzEulerLagrangeInner}
    \int\sb{\br\cap\{u-\phi>2\theta\}}%
    \!\!\!\!\!\!\!\!\!\!\!\!%
      \langle M(\diff u), \diff u-\diff\phi\rangle = 0.
  \end{displaymath}
   
  By subtracting the latter equation to the earlier one we get a
  contradiction:
  \begin{displaymath}
    \begin{split}
      0\ge%
      \!\!\!\!\!\!\!\!\!\!\!\!%
      \int\sb{\br\cap\{u-\phi>2\theta\}}%
      \!\!\!\!\!\!\!\!\!\!\!\!%
      -\langle M(\diff u)-M(\diff\phi), \diff u-\diff\phi\rangle \ge%
          \int\sb{\bt}\theta^2,
    \end{split}
  \end{displaymath}
  (here the first inequality follows from the monotonicity of $M$).
\end{proof}

We now introduce some further assumptions on the approximating
sequence $(\fa)\sb{\adom}$ considered in Proposition
\ref{basicFactsAboutApproxForU} and prove a couple
of preliminary results.

\begin{thm}\label{viscosityApproxForU}
Assume that the $f\sb{\a}$'s satisfy \eqref{eq:cond:fConvex}-\eqref{eq:cond:fRegularity} and
\begin{equation}
  \label{hp:regularityInTheLargeSet}
f'\sb{\a}(0)=0, \ n\in\NN,
\end{equation}
and let $\ua$ be
  the solution of problem \eqref{varpb} for $\ja$ in \eqref{eq:approxFunctionalForU}.
\par
Then $\ua$ is a viscosity solution in $\Omega$ of
  \begin{equation}\label{eq:apprnondiv}
    -\big[\aaa(\gm{u})-1\big]\lp\sb{\infty}u - \gm{u}^2\lp u = 0,
  \end{equation}
  where $\aaa$ is given by \eqref{aalfa}.
\end{thm}
\begin{proof}
It is easy to see that --- owing to \eqref{hp:regularityInTheLargeSet}
--- the Euler-Lagrange equation for the functional
\eqref{eq:approxFunctionalForU} is
  \begin{equation}
    \label{eq:eulerLagrangeForUp}
    \dv\big(\Aa(\diff u)\big)=0,
  \end{equation}
where $\fun{\Aa}{\Rn}{\Rn}$ is
the monotone operator defined by
\begin{equation}
  \label{def:Aa}
  \Aa(\vv) = \begin{cases}
                 \fa'(|\vv|)\dfrac{\vv}{|\vv|} & \vv\neq0,\\
                 0                       & \vv=0.
           \end{cases}
\end{equation}
(See also Lemma \ref{bridgeToMinimizers}.)
Then apply Lemma \ref{correctedViscousFormForGeneralWithModulus}, with
$M = \Aa$.
\end{proof}

 When $N=2$ (and $\Omega$ is simply connected), we can always define a
 stream function for each $\ua$.
 \begin{thm}
\label{th:va}
   Let $\Omega\subset\Rr$ be simply connected and let the assumptions
   of Theorem \ref{viscosityApproxForU} be in force. 
\par
Then the
   following assertions hold:
 \par
(i) for every $\adom$ there exists a
     unique Lipschitz continuous solution in the sense 
     of distributions $\va$ of the system
     \begin{equation}\label{eq:defOfVp}
            \partial\sb{x}\va =-\fa'(\gm{\ua})\frac{\partial\sb{y}\ua}{\gm{\ua}}\,, \qquad
            \partial\sb{y}\va =\fa'(\gm{\ua})\frac{\partial\sb{x}\ua}{\gm{\ua}}\,,
     \end{equation}
     such that 
 \begin{equation}
       \label{eq:vpNullMean}
       \int\sb{\Omega}\va = 0;
     \end{equation}
\par
(ii)
if $\dfrac{\rho\fa''(\rho)}{\fa'(\rho)}$
     converges to a positive constant as $\rho\to 0^+$, then $\va$ is a viscosity
     solution of 
     \begin{equation}
\label{eq:va}
       -\big[1-\aaa(\ga'(\gm{v}))\big]\lp\sb\infty v - %
         \gm{v}^2\aaa(\ga'(\gm{v}))\lp v = 0,
     \end{equation}
     in $\Omega,$ where $\aaa$ is given by \eqref{aalfa} and $\ga'$ is
     the inverse function of $\fa'$.
 
 \end{thm}
 \begin{proof}
(i)  As a minimizer of the differentiable functional $\ja$, $\ua$ is a weak solution
     of the corresponding Euler-Lagrange equation. The latter statement corresponds
     to saying that
     for any $\fhi\in\wqcs$ for $q\ge2$, it holds that
     \begin{displaymath}
       \label{eq:eulerLagrangeForUpNew}
       \int\sb\Omega\langle \Aa\big(\diff\ua(x)\big), \diff\fhi (x) \rangle dx = 0,
     \end{displaymath}
     where $\Aa$ is defined in \eqref{def:Aa}.
     The previous equation 
     may be interpreted as the following
differential form
     \begin{equation*}
       \omega_\a= \begin{cases}
             \dfrac{\fa'(|\diff\ua|)}{|\diff\ua|}%
               (-\partial\sb{y}\ua{}\,dx + \partial\sb{x}\ua{}\, dy) %
               & |\diff\ua|\neq0,\\
             0 %
               & \mathrm{elsewhere}.
             \end{cases}
     \end{equation*}
     being closed, as a form belonging to
     $L\sp{q}(\Omega)$, for $q\ge2$.\cri
     We required the domain $\Omega$ to be simply connected. Thus
     (see ~\cite[Lemma 3.2.1]{Schwarz}) we can
     integrate $\omega_\a$ to obtain a function $\va\in\wq$ such that
     \eqref{eq:defOfVp} holds.
The function $\va$ is not completely defined by the
     condition \eqref{eq:defOfVp}: we are left with the choice of a
     constant to add. We choose such costant so that \eqref{eq:vpNullMean} holds.
   \par
We observe that since \eqref{BdGrads} holds and 
$|\diff\va| = \fa'(|\diff\ua|)$, we can conclude
\begin{equation}
\label{BdGradsv}
|\diff\va|\le\fa'(Q)\ \mbox{ in } \ \Omega.
\end{equation}
\par
(ii)   Since (i) holds, $\va$ is a weak
       solution of
       \begin{equation*}
         \label{eq:correctedFormForVp}
         \dv\big(\Ba(\diff v)\big)=0
       \end{equation*}
       where $\Ba(p)$ is the monotone operator
       \begin{equation*}
         \label{def:B}
         \Ba(\vv) = \begin{cases}
           \ga'(|\vv|)\dfrac{\vv}{|\vv|} & \vv\neq0\\
           0                       & \vv=0,
         \end{cases}
       \end{equation*}
       which happens to be the inverse of $\Aa$.

       Then, owing to the stated assumptions on $\aaa$, the proof
       follows the lines of that of Lemma
       \ref{correctedViscousFormForGeneralWithModulus}, by observing
       that, due to the fact that $\ga'$ is the inverse of $\fa'$, then
       \begin{equation*}
         \dfrac{r\ga''(r)}{\ga'(r)}= \frac1{\aaa(\ga'(r))}.
       \end{equation*}
 \end{proof}


\par
Now, we want to take the limit in \eqref{eq:defOfVp}.

\begin{thm}\label{mainResultForSys}
Let \eqref{bsc} be in force and assume that $(\fa)\sb{\adom}$ is a sequence of strictly convex functions
converging uniformly to $f$ on $\Rhc$.  
\par
Let $u$
(resp. $\ua$) be the unique solution of \eqref{bsc}-\eqref{varpb} for $\jj$
(resp. for $\ja$ in \eqref{eq:approxFunctionalForU}) and let $\va$ be defined as in
Theorem \ref{th:va}. Also, assume that the gradients $\na\ua$ converge to $\na u$
almost everywhere in $\Omega.$
\par
Then $(\va)\sb{\adom}$ contains a subsequence which converges uniformly on 
$\overline{\Omega}$ to a function $v\in W^{1,\infty}(\Omega)$ and the pair $(u,v)$ satisfies
the system \eqref{invstream} almost everywhere in $\Omega.$
\end{thm}

\begin{proof}
Consider the functionals
\begin{eqnarray*}
&& \cF(u,v)=\ints_\Omega \{f(\gm{u})+g(\gm{v})-(u_x\, v_y-u_y\, v_x)\}\, dx dy, \\
&& \cF_\a(u,v)=\ints_\Omega \{\fa(\gm{u})+\ga(\gm{v})-(u_x\, v_y-u_y\, v_x)\}\, dx dy.
\end{eqnarray*}
From \eqref{eq:defOfVp} it is clear that $\cF_\a(\ua,\va)=0.$
\par
Now, observe that, by \eqref{BdGradsv} and the
uniform convergence of the $\ga$'s, the gradients of the $\va$'s are uniformly
bounded and, since \eqref{eq:vpNullMean} holds for every $\a\in\NN,$ the $\va$'s
satisfy on $\overline{\Omega}$ the assumptions of Ascoli-Arzel\`a's theorem. Thus, 
$(\va)\sb{\adom}$ contains a subsequence (that we will still denote by $(\va)\sb{\adom}$) 
which converges uniformly on 
$\overline{\Omega}$ to a function $v\in W^{1,\infty}(\Omega)$ and,
by the boundedness of $(\va)\sb{\adom}$ in $W^{1,\infty}(\Omega),$ we can
always assume that $(\va)\sb{\adom}$ weakly converges to $v$ in any 
$W^{1,p}(\Omega), \, p>1.$
The latter property implies that
$$
\lim_{n\to\infty}\ints_\Omega (\pa_x\ua\, \pa_y\va-\pa_y\ua\,\pa_x\va)\}\, dx dy=
\ints_\Omega (u_x\, v_y-u_y\, v_x)\}\, dx dy
$$
--- since the gradients $\na\ua$ are bounded and are assumed to converge a.e. to $\na u$ ---
and also that
$$
\liminf_{n\to\infty}\ints_\Omega \ga(\gm{\va})\, dx dy\ge \ints_\Omega g(\gm{v})\, dx dy,
$$
by the uniform 
convergence of $\ga$ to $g$ and the convexity of $g.$ 
Therefore, we can infer that
$$
\cF(u,v)\le\liminf_{n\to\infty}\cF_\a(\ua,\va)=0.
$$
\par
On the other hand, by the very definition \eqref{eq:fenchel} of $g,$ the following inequalities hold almost everywhere in $\Omega$:
$$
f(\gm{u})+g(\gm{v})-(u_x\, v_y-u_y\, v_x)\ge \gm{u}\,\gm{v}-(u_x\, v_y-u_y\, v_x)\ge 0;
$$
thus, $\cF(u,v)=0$ and hence
$$
f(\gm{u})+g(\gm{v})-(u_x\, v_y-u_y\, v_x)=\gm{u}\,\gm{v}-(u_x\, v_y-u_y\, v_x)= 0
$$
almost everywhere in $\Omega.$ These couple of equalities then yield \eqref{invstream}
at once.
\end{proof}
\begin{rmk}
The a.e. convergence of the gradients $\na\ua$ assumed in Theorem \ref{mainResultForSys} can be
obtained at least in two fashions. 
\par
The former consists in applying either \cite[Thm. 1]{evansGariepyWeakConv} or
\cite[Thm. 2, pg. 21]{evansWeakConvergence} (by possibly restricting our assumptions on $f$).
\par
The latter consists in an adaptation of the arguments used in \cite{magnaniniTalenti2}: since 
$\ua$ is a solution of \eqref{eq:apprnondiv}, it also satisfies the Bernstein's inequality
\begin{equation*}
   |\na^2 \ua|^2 \le 
     -\det(\na^2 \ua)\,\big[\aaa(\gm{\ua}) + \aaa(\gm{\ua})^{-1}\bigl],
  \end{equation*}
which, by the properties of the coefficients $\aaa,$  implies the bound
$$
\ints_\Omega \aaalim(\gm{\ua})\,|\na^2 \ua|^2\, dx dy\le C_K,
$$
where $K$ is any compact subset of $\Omega$ and $C_K$ is a constant depending on $K.$
\par
This last inequality provides the expected compactness of the sequence $(\ua)_{\a\in\NN}.$ 
\end{rmk}

\medskip
We conclude this section by showing that the functions $u$ and $v$ determined in
Theorem \ref{mainResultForSys} are solutions of second-order degenerate elliptic equations.
 
\begin{thm}\label{mainResultForU}
   Assume the \bsc{} \eqref{bsc} is in force and let $\f$ satisfy \eqref{eq:cond:f} and
   \eqref{def:alphaToInfty}.
\par
   Then the solution $u$ of problem \eqref{varpb} is a viscosity
   solution of
   \begin{equation}
     -\big[\aaalim(\gm{u})-1\big]\lp\sb\infty u - %
       \gm{u}^2\lp u =0\quad\text{ in }\Omega,
   \end{equation}
where $\aaalim$ is given by \eqref{alphaIntro}.
  \end{thm}
 \begin{proof}
   We can always approximate $\f$ by a sequence of lagrangeans $\fa$
   such that: 
\begin{itemize}
\item[(a)] the $\fa$'s converge to $\f$ uniformly on 
     $\Rhc$; 
\item[(b)] the functions $\aaa(|p|)p\sb{i}p\sb{j}$, with
     $\aaa$ given by \eqref{aalfa}, converge to
     $\aaalim(|p|)p\sb{i}p\sb{j}$, $i,j=1\ldots N$, uniformly on any
     compact subset of $\Rn$ (see also Remark \ref{practicalApprox}).
\end{itemize}
\par
 By Proposition \ref{basicFactsAboutApproxForU}, the solutions
     $\ua$ of the minimum problem \eqref{varpb} rephrased in terms of
     $\ja$ converge to $u$ uniformly in $\overline{\Omega}$. The
     conclusion then follows from
     ~\cite[Lemma 6.1]{viscosityUserGuide}.
 \end{proof}

 \begin{thm}\label{mainResultForV}
   Let $\f$ satisfy \eqref{eq:cond:f} and
   \eqref{def:alphaToInfty} and suppose $f$ is approximated, uniformly on $[0,\infty),$
by a sequence of functions $\fa$ which obey \eqref{eq:cond:fConvex}, \eqref{eq:cond:fRegularity}
and \eqref{hp:regularityInTheLargeSet}.
\par
Assume that the functions $\aaa$ satisfy the assumptions of 
Theorem \ref{th:va}  and that $\aaa\circ g'_n$ converge to a function $\beta$
uniformly on the compact subsets of $\Rhc$. 
\par  
Then the sequence $(\va)\sb{\adom}$ contains a subsequence that
   converges uniformly in $\Omega$ to a function a
   viscosity solution $v$ of
   \begin{equation}
\label{eq:limpV}
 -\big[1-\beta(|\na v|)\big]\,\Delta_\infty v - 
       |\na v|^2\,\beta(|\na v|)\,\Delta v=0 \qquad\text{ in }\Omega.
   \end{equation}
 \end{thm}
 \begin{proof}
   By Theorem \ref{th:va},  each $\va$ satisfies \eqref{eq:va}; \eqref{BdGradsv} and the properties of the $\fa$'s
imply that the
   sequence $(\va)\sb{\adom}$ satisfies the assumptions of Ascoli-Arzel\`a's
   theorem. Therefore the conclusion follows again by applying 
   ~\cite[Lemma 6.1]{viscosityUserGuide} to any uniformly converging
   subsequence of $(\va)\sb{\adom}$. 
 \end{proof}
 \begin{cor}\label{corollaryForV}
   If the coefficients $\aaa$ also converge to $\aaalim$ uniformly on
   every compact subset of $\Rhc$, then the function $v$ of Theorem
   \ref{mainResultForV} is a viscosity solution of 
   \begin{equation}
     \label{eq:viscosityEqForV}
     -\big[1-\aaalim(\g'(\gm{v}))\big]\lp\sb\infty v -%
       \gm{v}^2\aaalim(\g'(\gm{v}))\lp v=0\qquad\text{ in }\Omega,
   \end{equation}
   where $\g(r)=\sup\{r\rho-\f(\rho): \rho\ge0\}$.
 \end{cor}
\begin{proof}
It is enough to compute the function $\beta$. Since $g_n(r)$
always converge to $g'(r),$ we obtain that $\beta(r)=\aaalim(g'(r))$
at once.
\end{proof}
\begin{rmk}\label{practicalApprox}
   Sequences of $(\fa)_{\a\in\NN}$ that satisfy the assumptions
   mentioned in the statement of Theorem \ref{mainResultForV} can be
   constructed in various fashions. A convenient way is to modify
$f$ only in a neighborhood of $\rho=0.$
\par
   Here we give two examples; we set
   $\fa'(\rho)=\multa(\rho)\f'(\rho)$, with $\multa(\rho)\equiv1$ for
   $\rho\ge1$, while
   \begin{enumerate}[(a) ]
   \item $\multa^a(\rho)= 1-(1-\rho^{\ainv})^{\a}$ for $\rho\in[0,1]$;
   \item $\multa^b(\rho)= 1-(1-\rho^s)^{\a}$ for $\rho\in[0,1]$
   \end{enumerate}
where $s>0.$
\par
   With the obvious notations it easy to see that the $\aaa\sp{a}$'s
   converge uniformly to $\aaalim$ on $\Rhc$, so that Corollary
   \ref{corollaryForV} 
   applies. In particular, $\aaalim\circ g'\equiv 0$ on $[0, f'(0)]$
so that, for $|\na v|\le f'(0),$ \eqref{eq:viscosityEqForV} reads thus: $\Delta_\infty v=0.$
\par
The $\aaa\sp{b}$'s instead converge uniformly in the compact subsets of $(0,\infty)$
but not on those of $[0,\infty).$  Straightforward computations show that the function $v$
   of Theorem \ref{mainResultForV} satisfies \eqref{eq:limpV}, where
   \begin{equation*}
    \bbalim(r)=%
         \begin{cases}
           -s\,\dfrac{f'(0)-r}{r}\log[1-r/f'(0)] &0\le r\le f'(0),\\
           \aaalim(g'(r)) & r>f'(0).
         \end{cases}
   \end{equation*}
 \end{rmk}

\begin{acknowledgements}
The authors are indebted with Prof. G. Talenti for the many helpful
discussions and, in particular, for inspiring the use of functional
$\mathcal{F}(u,v)$ in Theorem \ref{mainResultForSys}.

\end{acknowledgements}

\bibliographystyle{plain}
\bibliography{biblio}

\end{document}